\documentclass[12pt]{amsart}
\usepackage{enumerate,amsmath,amssymb,bm}
\usepackage{url}
\usepackage{xspace}

\usepackage[margin=1in]{geometry}

\DeclareMathOperator{\pnt}{\raise 0.5mm \hbox{\large\textbf{.}}}

\newcommand{\note}[2][ ]{}

\newtheorem{theorem}{Theorem}
\newtheorem{lemma}[theorem]{Lemma}

\newtheorem{problem}[theorem]{Problem}

\theoremstyle{definition}
\newtheorem{remark}[theorem]{Remark}

\newtheorem{question}[theorem]{Question}

\usepackage[pdfauthor={William J. keith and Fabrizio Zanello},
            pdftitle={On the number of even values of an eta-quotient},
            pdfsubject={enumerative combinatorics},
            pdfkeywords={regular partitions, parity},
            pdfproducer={Latex with hyperref},
            pdfcreator={latex->dvips->ps2pdf},
            pdfpagemode=UseNone,
            bookmarksopen=false,
            bookmarksnumbered=true]{hyperref}

\begin{document}
\title[On the number of even values of an eta-quotient]{On the number of even values of an eta-quotient}
\author{Fabrizio Zanello} \address{Department of Mathematical Sciences\\ Michigan Tech\\ Houghton, MI  49931}
\email{zanello@mtu.edu}

\maketitle

\begin{abstract}
The goal of this note is to provide a general lower bound on the number of even values of the Fourier coefficients of an arbitrary eta-quotient $F$, over any arithmetic progression. Namely, if $g_{a,b}(x)$ denotes the number of even coefficients of $F$ in degrees $n\equiv b$ (mod $a$) such that $n\le x$, then we show that $g_{a,b}(x) / \sqrt{x}$ is unbounded for $x$ large.

Note that our result is very close to the best bound currently known even in the special case of the partition function $p(n)$ (namely, $\sqrt{x}\log \log x$, proven by Bella\"iche and Nicolas in 2016). Our argument substantially relies upon, and generalizes, Serre's classical theorem on the number of even values of $p(n)$, combined with a recent modular-form result by Cotron \emph{et al.} on the lacunarity modulo 2 of certain eta-quotients.

Interestingly, even in the case of $p(n)$ first shown by Serre, no elementary proof is known of this bound. At the end, we propose an elegant problem on quadratic representations, whose solution would finally yield a modular form-free proof of Serre's theorem.
\end{abstract}

\section{Introduction and preliminaries}

The goal of this brief note is to present a general result on the longstanding problem of estimating the number of even values of the Fourier coefficients of arbitrary eta-quotients (see below for the relevant definitions). In fact, we will do so over any arithmetic progression. Namely, denoting by $g_{a,b}(x)$ the number of even coefficients of an eta-quotient $F$ in degrees $n\equiv b$ (mod $a$) such that $n\le x$, in Theorem  \ref{main} we show that
$$\frac{g_{a,b}(x)}{\sqrt{x}}$$
is unbounded for $x$ large.

Our paper was originally motivated by the  preprint \cite{zhe}, which asked whether a specific eta-quotient assumes infinitely many even, and infinitely many odd values over any arithmetic progression. (See Conjecture 2 of the arXiv version v3 of \cite{zhe}, which has since been updated because of a mistake in one of the proofs, unrelated to our own paper. We thank the author for informing us.) Our Theorem \ref{main} positively answers the even part of that conjecture as a very special case; for the odd part, see Question \ref{odd} at the end of this note, again in the much broader framework of arbitrary eta-quotients.

The proof of Theorem \ref{main} is substantially based upon, and generalizes, Serre's classical theorem \cite{ser} on the parity of the ordinary partition function $p(n)$ (in fact, of a broader class of functions) over any arithmetic progression, combined with a recent result by Cotron \emph{et al.} \cite{CMSZ} on the lacunarity modulo 2 of eta-quotients satisfying a suitable technical assumption. The use of the latter result, which implicitly requires the theory of modular forms, will constitute the only non-elementary portion of our argument. 

We note that even for the special case of $p(n)$, Serre's proof also required modular forms in an essential fashion. In fact, while it is easy to see that the number of even values of the partition function for $n\le x$ has order \emph{at least} $\sqrt{x}$, no elementary proof is known to date that this number grows \emph{faster} than $\sqrt{x}$. At the end of this paper, we propose a problem, phrased entirely in terms of quadratic representations, whose solution would finally yield a modular form-free proof of Serre's theorem for $p(n)$.

We first briefly recall the main definitions. We refer the reader to, e.g., \cite{KeZa} and its references for any unexplained terminology. Set $f_j=f_j(q) = \prod_{i\ge 1} (1-q^{ji})$. Then an \emph{eta-quotient} is a quotient of the form
\begin{equation}\label{etaq}
F(q)=\frac{\prod_{i=1}^uf_{\alpha_i}^{r_i}}{\prod_{i=1}^tf_{\gamma_i}^{s_i}},
\end{equation}
for integers $\alpha_i$ and $\gamma_i$ positive and distinct, $r_i, s_i > 0$, and $u,t\ge 0$. (Note that, for simplicity, here we omit the extra factor of $q^{(1/24) \left(\sum \alpha_i r_i - \sum \gamma_i s_i\right)}$ that appears in some definitions of $F$, since this factor is irrelevant for the asymptotic estimates of this paper.)

We say that $F(q)=\sum_{n\ge 0}a_nq^n$ is \emph{odd with density} $\delta$ if the number of odd coefficients $a_n$ with $n\le x$ is asymptotic to $\delta x$, for $x$ large. Further, $F$ is \emph{lacunary modulo 2} if it is odd with density zero (equivalently, if its number of odd coefficients is $o(x)$).

One of the best-known instances of an eta-quotient (\ref{etaq}) is arguably
$$\frac{1}{f_1}=\sum_{n\ge 0} p(n)q^n.$$
Understanding the parity of $p(n)$ is a horrendously difficult and truly fascinating problem, which has historically attracted the interest of the best mathematical minds. A classical conjecture by Parkin-Shanks \cite{Calk,PaSh} predicts that $p(n)$ is odd with {density} $1/2$ (see \cite{JKZ,JZ,KeZa} for generalizations of this conjecture). However, the best bounds available today, obtained after a number of incremental results, only guarantee that the even values of $p(n)$ are of order at least $\sqrt{x}\log \log x$ \cite{BelNic}, and the odd values at least $\sqrt{x}/\log \log x$ \cite{BGS}.

Thanks to theorems by Ono \cite{Ono2} and Radu \cite{Radu}, we also know that $p(n)$ assumes infinitely many odd, and infinitely many even values over any arithmetic progression. In fact, as we mentioned earlier, Serre's result \cite{ser} established that the number of even values of $p(n)$ for $n\le x$, $n\equiv b$ (mod $a$) grows faster than $\sqrt{x}$, for any choice of $a$ and $b$. However, it is reasonable to believe, as a generalization of the Parkin-Shanks conjecture (see \cite{NiSa}), that $p(n)$ is even with density $1/2$ {over any arithmetic progression}. For a broader set of conjectures on the parity of eta-quotients, including their behavior over arithmetic progressions, see our recent paper with Keith (\cite{KeZa}, Conjecture 4).

In view of the above, it appears that our bound of Theorem \ref{main} --- which holds in full generality for any eta-quotient, and is very close to the best known result even in the case of $p(n)$ \cite{BelNic} --- might be hard to improve significantly with the existing technology.

\section{The bound}

We first need a recent theorem by Cotron \emph{et al.}, which we restate in the following terms:
\begin{lemma}[\cite{CMSZ}, Theorem 1.1]\label{cot}
Let $F(q)=\frac{\prod_{i=1}^uf_{\alpha_i}^{r_i}}{\prod_{i=1}^tf_{\gamma_i}^{s_i}}$ be an eta-quotient as in (\ref{etaq}), and assume that
$$\sum_{i=1}^u \frac{r_i}{\alpha_i} \ge \sum_{i=1}^t s_i\gamma_i.$$
Then $F$ is lacunary modulo 2.
\end{lemma}

We are now ready for the main result of this note. In what follows, given two power series $A(q) = \sum_{n\ge 0} a(n) q^n$ and $B(q) = \sum_{n\ge 0} b(n) q^n$, if we write $A(q) \equiv B(q)$ we always mean that $a(n) \equiv b(n)$ (mod 2), for all $n$.
\begin{theorem}\label{main}
Let $F(q)=\sum_{n\ge 0}a_nq^n$ be an eta-quotient as in (\ref{etaq}), and denote by $g_{a,b}(x)$ the number of even values of $a_n$ over the arithmetic progression $n\equiv b$ (mod $a$), for $n\le x$. Then
$$\lim_{x\rightarrow \infty} \frac{g_{a,b}(x)}{\sqrt{x}}=\infty.$$
\end{theorem}

\begin{proof}
Let
$$F(q)=\frac{\prod_{i=1}^uf_{\alpha_i}^{r_i}}{\prod_{i=1}^tf_{\gamma_i}^{s_i}}=\sum_{n\ge 0}a_nq^n$$
be as in the statement. Since
$$\frac{q^b}{1-q^a}=\sum_{j\ge 0}q^{b+ja},$$
it is clear that the coefficients of the series $G$ defined by
$$G(q)=\frac{q^b}{1-q^a}+F(q)$$
coincide with those of $F$ except precisely in degrees $n\equiv b$ (mod $a$), where they differ by 1. In particular, $F$ is even in any degree $n\equiv b$ (mod $a$) if and only if $G$ is odd in that degree.

Now fix a positive integer $d$. By definition of $G$, we have the identity:
\begin{equation}\label{ddd}
G(q) \cdot f_a^{2^d} = q^b \cdot \frac{f_a^{2^d}}{1-q^a} + f_a^{2^d} \cdot \frac{\prod_{i=1}^uf_{\alpha_i}^{r_i}}{\prod_{i=1}^tf_{\gamma_i}^{s_i}}.
\end{equation}
Using the reduction modulo 2 of Euler's Pentagonal Number Theorem (see, e.g., \cite{Andr}),
$$f_1 \equiv \sum_{n \in \mathbb Z} q^{n(3n-1)/2},$$
we obtain:
$$f_a^{2^d} \equiv \left(\sum_{n \in \mathbb Z} q^{a n(3n-1)/2}\right)^{2^d}\equiv \sum_{n \in \mathbb Z} q^{2^d\cdot a n(3n-1)/2}.$$

It follows by standard computations that, for $x$ large,  the number of odd coefficients of $f_a^{2^d}$ in degrees $n\le x$ is asymptotic to
$$\frac{c_0\sqrt{x}}{2^d},$$
where
$$c_0=\frac{2 \sqrt{2}}{a\sqrt{3}}.$$
Further, all  odd coefficients appear in degrees $n\equiv 0$ (mod $a$).

From the Pentagonal Number Theorem, we also deduce the modulo 2 identity:
$$\frac{f_1}{1-q}\equiv (1+q+q^2+q^5+q^7+q^{12}+q^{15}+q^{22}+\dots)(1+q+q^2+q^3+\dots)$$
\begin{equation}\label{23}
\equiv 1+(q^2+q^3+q^4)+(q^7+\dots+q^{11})+(q^{15}+\dots+q^{21})+\dots
\end{equation}
Note that the last series alternates strings of consecutive powers with coefficient 1 to strings (omitted) of consecutive powers with coefficient 0, where a new string begins any time a degree is a generalized pentagonal number.

Given this, it is easy to see that, for $x$ large, the number of odd coefficients of $f_1/(1-q)$ (or equivalently, the number of 1s appearing in (\ref{23})) in degrees $n\le x$ is asymptotic to $2x/3$.

Moreover, since $f_1^{2^d}\equiv f_{2^d}$  (mod 2), an entirely similar argument gives that the corresponding asymptotic value for the odd coefficients of
$$\frac{f_1^{2^d}}{1-q}$$
is again $2x/3$. Thus, if we replace $q$ with $q^a$, it is clear that the number of odd coefficients of
$$\frac{f_a^{2^d}}{1-q^a}$$
in degrees $n\le x$ is asymptotic to $c_1x$, with $c_1=2/(3a)$.

It follows that the number of odd coefficients of the first term on the right side of (\ref{ddd}),
$$q^b \cdot \frac{f_a^{2^d}}{1-q^a},$$
is asymptotic to $c_1x$. Note that these coefficients all appear in degrees $n\equiv b$ (mod $a$).

Now consider the second term on the right side of (\ref{ddd}), namely
\begin{equation}\label{444}
f_a^{2^d} \cdot \frac{\prod_{i=1}^uf_{\alpha_i}^{r_i}}{\prod_{i=1}^tf_{\gamma_i}^{s_i}}.
\end{equation}
By Lemma \ref{cot}, wo obtain that (\ref{444}) is lacunary modulo 2 whenever
$$\frac{2^d}{a}+\sum_{i=1}^u \frac{r_i}{\alpha_i} \ge \sum_{i=1}^t s_i\gamma_i,$$
or equivalently,
$$2^d \ge a \left(\sum_{i=1}^t s_i\gamma_i - \sum_{i=1}^u \frac{r_i}{\alpha_i}\right).$$
Thus, the lacunarity of (\ref{444}) is guaranteed for all integers $d$ large enough.

Putting the above together, for any large integer $d$, the number of odd coefficients on the right side of (\ref{ddd}) in degrees $n\le x$ is asymptotic to
$$c_1x+o(x),$$
or simply to $c_1x$. Asymptotically, again $c_1x$ of these odd coefficients are in degrees $n\equiv b$ (mod $a$).

Recall that the number of odd coefficients of $f_a^{2^d}$, which is the second factor on the left side of (\ref{ddd}), was shown to be asymptotic to $c_0\sqrt{x}/2^d$ (where the constant $c_0$ is independent of $d$). Also, all such coefficients appear in degrees $n\equiv 0$ (mod $a$).

We conclude that, for $n\equiv b$ (mod $a$), $n\le x$, the number of odd coefficients of the first factor, $G$, on the left side of (\ref{ddd}) --- or equivalently, the number $g_{a,b}(x)$ of even coefficients of the original eta-quotient $F$ --- must satisfy:
$$g_{a,b}(x) \ge c_2\cdot \frac{c_1x}{\left(c_0\sqrt{x}\right)/2^d}=c_3\cdot 2^d\sqrt{x},$$
for a suitable positive constant $c_2$ and for $x$ large, where
$$c_3=\frac{c_2c_1}{c_0}>0$$
is independent of $d$. Thus,
$$\frac{g_{a,b}(x)}{\sqrt{x}}\ge c_3\cdot 2^d,$$
for $x$ large. Since this is true for all $d$ large, the theorem follows.
\end{proof}

\section{Questions for future research}

As we mentioned earlier, Lemma \ref{cot}, which was a key ingredient in the proof of Theorem \ref{main}, made an essential use of modular forms. We are not aware of a modular form-free proof of our result, even for $p(n)$. In fact, Serre's argument \cite{ser} that the even values of $p(n)$, for $n\le x$, grow faster than $\sqrt{x}$ also relied on modular forms; to be precise, it can be seen that their use may essentially be limited to showing that $f_1^{2^d-1}$ is lacunary modulo 2 for infinitely many values of $d$.

We now propose a problem, elegantly stated in terms of quadratic representations, which is equivalent to the lacunarity modulo 2 of $f_1^{2^{2d}-1}$. (For brevity's sake, we omit the proof of this equivalence, which employs the Jacobi triple product identity and other elementary tools.) Thus, a direct proof of Problem \ref{pro} would lead, as a byproduct, to the first modular form-free proof of Serre's theorem in the case of $p(n)$.
\begin{problem}\label{pro}
For a positive integer $d$, consider the polynomial
$$T_d \left(x_1,\dots,x_{2^{d-1}}\right)=\sum_{i=1}^{2^{d-1}}2^{d}x_i^2-(2i-1)x_i,$$
and let
$$R_{d}(n)=\# \left\{\left(x_1,\dots,x_{2^{d-1}}\right)\in \mathbb Z^{2^{d-1}}:T_d \left(x_1,\dots,x_{2^{d-1}}\right)=n \right\}.$$
\noindent
Show that, for any $d\ge 1$,
$$\# \left\{n\le x:R_{d}(n){\ }\text{is odd} \right\}=o(x).$$
\end{problem}
\noindent
\begin{remark}\label{qr}
\begin{enumerate}
\item As we saw earlier, the lacunarity modulo 2 of $f_1^{2^{2d}-1}$, which is equivalent to the statement of Problem \ref{pro} for $d$, is already known via modular forms \cite{ser}. In fact, when $d\ge 2$, with more work (and more modular forms) one can estimate exactly that, for $x$ large, $\# \left\{n\le x :R_{d}(n){\ }\text{is odd} \right\}$ is asymptotic to
$$c_1x(\log \log x)^{c_2}/\log x,$$
for suitable constants $c_1$ and $c_2$ depending on $d$ (see \cite{NS,ser1} for details).

\item In order to reprove Serre's theorem on the even values of $p(n)$, one can show that, in fact, it suffices to solve Problem \ref{pro} only \emph{for infinitely many values of $d$}, and \emph{over any arithmetic progression $cn+r$} in lieu of $n$, provided that $c$ grow slower than $4^d$.

\item As an illustration, when $d=1$, it is easy to see that $R_{1}(n)$ is odd (it equals 1) if and only if $n=\binom{a+1}{2}$. Thus, the estimate $o(x)$ is trivial. When $d=2$, $R_{2}(n)$ is odd precisely when the number of representations of $n$ as $4\binom{a+1}{2}+\binom{b+1}{2}$ is odd. Since by a classical result of Landau \cite{Landau,ser1}, a binary quadratic form is lacunary over the integers, this is \emph{a fortiori} true modulo 2, and the result again follows.
\end{enumerate}
\end{remark}

We conclude by posing a general question on the number of \emph{odd} values of arbitrary eta-quotients over arithmetic progressions.
\begin{question}\label{odd}
\emph{Let $F(q)=\sum_{n\ge 0}a_nq^n$ be an eta-quotient as in (\ref{etaq}), and assume $F$ is not constant modulo 2 over the arithmetic progression $n\equiv b$ (mod $a$). Is it true that the number of odd values of $a_n$ for $n\equiv b$ (mod $a$), $n\le x$, has always order at least $\sqrt{x}$?}
\end{question}

Note that the lower bound of Question \ref{odd} in general cannot be improved ($\sqrt{x}$ is well known to be the correct order for the odd values of, e.g., $f_1$ and $f_3$). However, the question is still open in many important instances; among others, for the generating function $1/f_1$ of $p(n)$, and for all of its positive powers $1/f_1^t$, which define the $t$-multipartition functions $p_t(n)$ \cite{BGS,JKZ,JZ}.

\section*{Acknowledgements} We thank William Keith for useful discussions on Serre's theorem. This work was partially supported by a Simons Foundation grant (\#630401).


\begin{thebibliography}{99}

\bibitem{Andr} G. Andrews: \lq \lq The Theory of Partitions,'' Encyclopedia of Mathematics and its Applications, Vol. II. Addison-Wesley, Reading, Mass.-London-Amsterdam (1976).

\bibitem{BGS} J. Bella\"iche, B. Green, and K. Soundararajan: \emph{Non-zero Coefficients of Half-Integral Weight Modular Forms Mod $\ell$}, Res. Math. Sci. \textbf{5} (2018), no. 1, Paper no. 6, 10 pp..

\bibitem{BelNic} J. Bella\"iche and J.-L. Nicolas: \emph{Parit\'e des coefficients de formes modulaires}, Ramanujan J. \textbf{40} (2016), no. 1, 1--44.

\bibitem{Calk} N. Calkin, J. Davis, K. James, E. Perez, and C. Swannack: \emph{Computing the integer partition function}, Math. Comp. \textbf{76} (2007), 1619--1638.

\bibitem{CMSZ} T. Cotron, A. Michaelsen, E. Stamm, and W. Zhu: \emph{Lacunary Eta-quotients Modulo Powers of Primes}, Ramanujan J.  \textbf{53} (2020), 269--284.

\bibitem{JKZ} S. Judge, W.J. Keith, and F. Zanello: \emph{On the Density of the Odd Values of the Partition Function}, Ann. Comb. \textbf{22} (2018), no. 3, 583--600.

\bibitem{JZ} S. Judge and F. Zanello: \emph{On the density of the odd values of the partition function, II: An infinite conjectural framework}, J. Number Theory \textbf{188} (2018), 357--370.

\bibitem{KeZa} W.J. Keith and F. Zanello: \emph{Parity of the coefficients of certain eta-quotients}, J. Number Theory \textbf{235} (2022), 275--304.

\bibitem{Landau} E. Landau: \emph{\"{U}ber die Einteilung der positiven ganzen Zahlen in vier Klassen nach der Mindestzahl der zu ihrer additiven Zusammensetzun erforderlichen Quadrate}, Arch. Math. Phys. (3) \textbf{13} (1908), 305--312.

\bibitem{NiSa} J.-L. Nicolas and A. S\'ark\"ozy: \emph{On the parity of partition functions}, Illinois J. Math. \textbf{39} (1995), no. 4, 586--597.

\bibitem{NS} J.-L. Nicolas and J.-P. Serre: \emph{Formes modulaires modulo 2: l'ordre de nilpotence des op\'erateurs de Hecke}, C.R. Acad. Sci. Paris, Ser. I \textbf{350} (2012), 343--348.

\bibitem{Ono2} K. Ono: \emph{On the parity of the partition function in arithmetic progressions}, J. Reine Angew. Math. \textbf{472} (1996), 1--15.

\bibitem{PaSh} T.R. Parkin and D. Shanks: \emph{On the distribution of parity in the partition function}, Math. Comp. \textbf{21} (1967), 466--480.

\bibitem{Radu} C.-S. Radu: \emph{A proof of Subbarao's conjecture}, J. Reine Angew. Math. \textbf{672} (2012), 161--175. 

\bibitem{ser1} J.-P. Serre: \emph{Divisibilit\'e de certaines fonctions arithm\'etiques}, L'Enseignement Math. \textbf{22} (1976), 227--260.

\bibitem{ser} J.-P. Serre: Appendix to: J.-L. Nicolas, I.Z. Ruzsa, and  A. S\'ark\"ozy: \emph{On the parity of additive representation functions}, J. Number Theory \textbf{73} (1998), no. 2, 292--317.

\bibitem{zhe} Q.-Y. Zheng: \emph{Distribution of partitions of $n$ in which no part appears exactly once}, preprint (arXiv:2205.03191; version v3 of May 12, 2022).

\end{thebibliography}
\end{document}